\documentclass[a4paper,12pt]{article}

\newenvironment{proof}{\hbox{}\vspace{-0.5cm} {\bf Proof:}}{\hfill $\Box$ \\}

\usepackage{lmodern}

\newtheorem{theorem}{Theorem}
\newtheorem{lemma}{Lemma}

\newtheorem{corollary}{Corollary}

\newtheorem{definition}{Definition}

\newtheorem{remark}{Remark}

\usepackage{comment}
\usepackage{amsmath}
\usepackage{amssymb}
\usepackage{amsfonts}
\usepackage{mathrsfs} 
\usepackage{graphicx}
\usepackage[dvipsnames]{xcolor}
\usepackage{ulem}
\usepackage{bm}

\usepackage{cleveref}

\newcommand{\R}{\mathbb{R}} % real numbers
\newcommand{\N}{\mathbb{N}} % integers
\newcommand{\Z}{\mathbb{Z}} % signed integers

\newcommand{\sol}{y} % solution to the PDE
\newcommand{\hilbert}{\mathscr{H}} % Hilbert space
\newcommand{\solset}{\mathscr{Y}} % set of solutions
\newcommand{\test}{\mathscr{C}} % test functions
\newcommand{\testcyl}{\mathrm{Cyl}}
\newcommand{\meas}{\mathscr{M}} % Borel measures

\newcommand{\scal}[2]{\langle{#1},{#2}\rangle} % scalar product

\newcommand{\didier}[1]{{\color{black} #1}} % didier's comments
\title{\bf Optimizing quasi-dissipative evolution equations with the moment-SOS hierarchy\footnote{D. Henrion's work was co-funded by the European Union
under the ROBOPROX project (reg. no. CZ.02.01.01/00/22 008/0004590). S. Marx's work was co-funded by the Agence National de la Recherche (ANR) under the ROTATION project (reg. no. ANR-24-CE48-0759). N. Seguin's work was co-funded by the Agence National de la Recherche (ANR) under the ANR HEAD project ANR-24-CE40-3260.}}

\begin{document}

\author{Saroj Prasad Chhatoi$^1$, Didier Henrion$^{1,2}$, Swann Marx$^3$, Nicolas Seguin$^4$}

\footnotetext[1]{CNRS; LAAS; Universit\'e de Toulouse, 7 avenue du colonel Roche, F-31400 Toulouse, France.}
\footnotetext[2]{Faculty of Electrical Engineering, Czech Technical University in Prague,
Technick\'a 2, CZ-16626 Prague, Czechia.}
\footnotetext[3]{Nantes Universit\'e, Ecole Centrale Nantes,
CNRS, LS2N, UMR 6004, F-44000 Nantes, France.}
\footnotetext[4]{IMAG, Inria, Universit\'e de Montpellier, CNRS, Montpellier, France.}

\date{Draft of \today}

\maketitle

\begin{abstract}
We prove that there is no relaxation gap between a quasi-dissipative nonlinear evolution equation in a Hilbert space and its linear Liouville equation reformulation {on probability measures}. In other words, strong and generalized solutions of such equations are unique in the class of measure-valued solutions. As a major consequence, non-convex numerical optimization over these non-linear partial differential equations can be carried out with the {infinite-dimensional} moment-SOS hierarchy with {global} convergence guarantees. This covers in particular all reaction-diffusion equations with polynomial nonlinearity.
\end{abstract}

\section{Introduction}

The {\it moment-SOS\footnote{SOS: sum of squares.} hierarchy} is a mathematical technology that allows to solve numerically with global optimality guarantees a large class of non-convex optimization\footnote{Optimization should here be interpreted broadly: finite- and infinite-dimensional, continuous and discrete versions of calculus of variations, optimal control and optimal transport problems on ordinary, stochastic and partial differential equations.} problems at the price of solving a family of convex relaxations (typically semidefinite optimization problems) of increasing size. The approach and its applications are described in \cite{hkl20}.

The first step in the approach consists of reformulating a non-convex non-linear optimization problem in a given domain as a linear problem in cones of measures supported on this domain. An archetypal example  is the Kantorovich linear reformulation of the non-linear Monge problem of optimal transport \cite{s15}. An important question is whether there is a {\it relaxation gap}, i.e. whereas the value of the linear problem on measures differ from the value of the original non-linear problem. In the absence of a relaxation gap, it can be shown (under standard assumptions) that the moment-SOS hierarchy converges, i.e. it generates a sequence of bounds which converges monotonically to the value of the original problems. Then a globally optimal solution can be approximated from the solutions of the successive convex relaxations, see e.g. \cite{mn24} for optimal transport.

For non-linear partial differential equations
(PDEs) and their control, an early attempt to use the linear measure formulation was reported in \cite{r93} for semi-linear elliptic equations. More recent attempts using the moment-SOS approach can be found in \cite{bc06,k15,mp20} for linear PDEs, and more recently in \cite{gf19,ft22} for non-linear PDEs. A fully general non-linear setup including optimization, calculus of variations and PDEs was described in \cite{khl22}. As far as we know, in these references, there is neither convergence guarantee of the proposed hierarchy, nor a proof of no relaxation gap.

In the context of scalar hyperbolic conservation laws, a particular class of nonlinear PDEs, no relaxation gap was ensured by introducing entropy inequalities \cite{mwhl20} (see also \cite{cmns24} for a similar result for conservation laws admitting parameters). 
In \cite{kr22} it was shown that there is no relaxation gap for scalar problems (i.e. when then  dimension of the domain or the dimension of the codomain is equal to one),
both for calculus of variations and for optimal control problems. Conversely, an example of a variational problem with relaxation gap is provided when the dimensions of the domain and of the codomain are greater than one.
It is also shown that in the presence of integral constraints, a relaxation gap may occur at any dimension of the domain and of the codomain. More recently, it was shown \cite{hkkr24} that under convexity assumptions, there is no relaxation gap for a broad class of variational and optimal control problems on non-linear PDEs. 
The question of the absence of a relaxation gap for specific classes of non-linear PDEs remains however widely open.

These results were obtained for measures supported on finite-dimensional subsets of the time, space, function and function gradient domains. There is however the possibility of reformulating optimization problems over nonlinear PDEs as linear problems on measures supported on infinite-dimensional functional spaces. This is the point of view adopted e.g. in \cite[Part III]{f99} for optimization with relaxed controls. Measures on infinite-dimensional spaces and statistical measure-valued solutions are also prominently appearing in \cite{fmrt01}
and in the subsequent work \cite{rt22} on Navier-Stokes equations. This notion of statistical solutions allows for sound numerical implementations, see e.g. \cite{flmw20}. These solutions seem to be related with the measure-valued solutions studied in \cite{hikv23} and approximated numerically in the context of the moment-SOS hierarchy. This motivated us to focus on reformulations with measures on infinite-dimensional spaces and to address the question of relaxation gap in this infinite-dimensional setup.

In this paper, we show that there is no relaxation gap for a broad class of non-linear PDEs defined by evolution equations on infinite-dimensional functional spaces, with an operator satisfying a quasi-dissipativity condition also known as {quasi-monotonicity} or $m$-accretive property, explained in details in \cite{crandall1971generation,showalter2013monotone,ik02}. The key ingredient in the proof is differential calculus with probability measures \cite{ags08}. No relaxation gap for this class of PDEs implies that we can use the infinite-dimensional moment-SOS hierarchy to solve optimization and optimal control problems on these equations, with convergence guarantees. It is of interest to mention \cite{al18} which also provides an alternative proof in the case of Hamiltonian PDEs, but without providing a numerical scheme nor making a clear link with the semigroup literature. To the best of our knowledge, this link is disclosed in our paper for the first time. 

{Existing numerical methods for non-linear PDEs focus on specific classes of PDEs. Except conservation laws for which sophisticated methods such as the wave front tracking is available  \cite{godlewski2013numerical}, these methods focus on semi-linear equations and provide a numerical scheme for the linear part which is proved to work with the nonlinear terms of the equation by using some fixed-point arguments. Classical numerical schemes for PDEs are based on space(-time) discretizations include the finite-difference method \cite{cfl67}, the finite-element method \cite{brenner2008mathematical}, the finite-volume method \cite{EGH} or discontinuous Galerkin methods \cite{c98}. We refer the interested reader to \cite{tadmor2012review} for a general overview of numerical schemes for nonlinear PDEs. Existing methods being based on fixed-point arguments and discretization (including implicit schemes), only local convergence can be guaranteed. Furthermore, when facing optimization problems involving PDEs, only local minima can be obtained with such methods.
	{See e.g. \cite{p22} for a simple example of a non-convex PDE optimal control problem with several local (and global) optima.}
	
	In contrast with all these methods, the moment-SOS hierarchy is designed to approximately compute {\it globally optimal solutions}. The main difference with finite-difference, finite-element and finite-volume methods consists in {discretizing measures via finitely many of their moments}, instead of discretizing the time and the space. Therefore, one may interpret this as a Galerkin method. However, the difference relies on the equation which is solved: indeed, while the Galerkin method focuses on solving (possibly non-linear) equations with functions as unknowns, the moment-SOS hierarchy is used in the context of (always linear) equations with measures as unknowns.
	
	More broadly speaking, the moment-SOS hierarchy has proved to be applicable in many other areas of dynamical systems such as optimal control \cite{lhpt08} or the computation of region of attraction \cite{hk14} for ordinary differential equations, and we hope that the results presented in this paper can pave the way for similar achievements for non-linear PDEs, i.e. optimal control (including boundary control) and region of attraction for nonlinear PDEs, as well as inverse problems.}

{The outline of the paper is as follows. In Section \ref{sec:quasimonotone} we introduce the non-linear PDEs under study, in the form of differential equations with quasi-dissipative operators evolving in a Hilbert space. A few examples are described in Section \ref{sec_examples}. In Section \ref{sec:linear} we introduce the Liouville equation, a linear measure reformulation of the non-linear PDE, as well as its relationship with various notions of solutions. Section \ref{sec:nogap} contains our two main results: first we state in Theorem \ref{nogap} that there exists a unique solution to the Liouville as soon as the generator is supposed to be quasi-dissipative; and second, as a consequence of the first result, we show a no relaxation gap result which states that for Dirac initial data, the unique solution to the Liouville equation is the Dirac at the solution of the non-linear PDE.
	In Section \ref{sec:hierarchy} we briefly describe the infinite-dimensional moment-SOS hierarchy. It is then illustrated numerically in Section \ref{sec:examples}. Further research directions and extensions are described in the concluding Section \ref{sec:conclusion}.}

\section{Quasi-dissipative evolution equations}\label{sec:quasimonotone}

Consider the evolution equation
\begin{equation}\label{ode}
\dot{\sol}(t) = f(\sol(t)), \quad \sol(0)=\sol_0, \quad {t \in [0,1]}
\end{equation}
with the time-dependent function $\sol : [0,{1}] \to \hilbert$, where $\hilbert$ is a given real Hilbert space equipped with a norm $|\cdot|$ and a scalar product $\langle \cdot,\cdot\rangle$, the dot denotes the time derivative, $f : D(f)\subset \hilbert \to \hilbert$ is a given nonlinear operator, with domain $D(f)$ densely defined in $\hilbert$, and $\sol_0 \in D(f)$ is a given initial condition. We further assume that $\hilbert$ forms a {rigged Hilbert space} (or a Gelfand triple) $\hilbert_1\subset \hilbert \subset \hilbert_{-1}$, where $\hilbert_1$ is supposed to be equal to $D(f)$, which is equipped with a norm $| z|$ for all $z\in \hilbert_1$, and where $\hilbert_{-1}$ is {defined as the topological dual to} $\hilbert_1$. A typical example is $\hilbert:=L^2(\mathbb{R})$ (square integrable functions), $\hilbert_1:=H^1(\mathbb{R})$ (functions with square integrable weak derivatives) and $\hilbert_{-1}:=H^{-1}(\mathbb{R})$ (dual space including distributions). In the case where $f$ is semilinear, i.e., $f(\sol):=A\sol + g(\sol)$, with  $A:D(A)\subset \hilbert \rightarrow \hilbert$ the generator of a strongly continuous semigroup, and $g$ a bounded operator from $\hilbert$ to $\hilbert$, the space $\hilbert_1$ is defined as $\hilbert_1:=D(A)$ (equal to $D(f)$ since $g$ is bounded) and $\hilbert_{-1}$ can be built \cite{tw09} as the completion of $\hilbert$ with respect to the norm $|(A-\rho I)^{-1} y|$ for all $y\in \hilbert$, where $\rho$ is an element of the resolvent of $A$.  We suppose furthermore that these Hilbert spaces are separable. {A solution of \eqref{ode} is denoted $\sol(t)$ or $\sol(t|\sol_0)$ if we want to emphasize the dependence on initial condition.} Finally, let $\solset$ denote a subset of $\hilbert$.

Before considering further properties of the operator $f$, we explain now what we mean by a solution of evolution equation \eqref{ode}. Indeed, there are several {non-equivalent} notions of solutions for infinite-dimensional systems. The first one, perhaps the most straightforward, is referred as strong solution \cite[Definition 2.2]{crandall1971generation}.

\begin{definition}[Strong solution]\label{defstrong}
	Let $\sol_0\in D(f)$. A function $\sol \in {\test([0,{1}],\hilbert)}$ is a strong solution to \eqref{ode} if it is Lipschitz continuous on compact subsets of $[0,{1}]$, differentiable almost everywhere, and equation \eqref{ode} is satisfied for almost every $t \in [0,{1}]$. 
\end{definition}

When facing lower regularity (i.e., with $\sol_0\in \hilbert$), a more suitable notion of solution should be introduced \cite[Definition, page 183]{showalter2013monotone}.

\begin{definition}[Generalized solution]\label{defgeneralized}
	Given $\sol_0\in \hilbert$, a generalized solution to \eqref{ode} is a function $y\in \test([0,{1}], \hilbert)$ for which there exists a sequence of initial conditions $(\sol_0^n)_{n\in\mathbb N}\in \hilbert_1$ converging to $\sol_0$ and strong solutions $(\sol^n)_{n\in \mathbb{N}}$ to \eqref{ode} converging to $\sol$, in the topology of $\hilbert$ and  $\test([0,{1}], \hilbert)$ respectively. %\nicolas{I think we should add the associated sequence of initial data $(y_0^k)_{k\in\mathbb{N}}$ in this definition}
\end{definition}

Another notion of interest, {particularly relevant for our purposes}, is the positive invariance of a set of solutions.

\begin{definition}[Positive invariance]\label{defsolset}
	Set $\solset\subset\hilbert$ is called positively invariant if $\sol_0 \in \solset$ implies $\sol(t) \in \solset$ for all $t \in [0,{1}]$.
\end{definition}

In this paper, we focus on evolution equations with operators satisfying a specific positivity condition.

\begin{definition}[Quasi-dissipative]\label{defqm}
	Operator $f$ is quasi-dissipative  on $\solset$ if there exists a constant $a\geq 0$ such that
	\begin{equation}\label{monotone}
	\scal{\sol_1-\sol_2}{f(\sol_1)-f(\sol_2)} \leq  a |\sol_1-\sol_2|^2
	\end{equation}
	for all $\sol_1,\sol_2 \in \solset$. 
\end{definition}

\begin{definition}[Maximality]
	Operator $f$ is maximal on $\solset$ if
	\begin{equation}\label{maximal}
	\forall h \in \solset, \: \exists \sol \in D(f), \: \sol=f(\sol)+h.
	\end{equation}
\end{definition}

{In the literature, quasi-dissipative operators are also called (up to a change of sign) quasi-monotone or $m$-accretive, see \cite{crandall1971generation,showalter2013monotone,ik02}.} 
If operator $f$ in \eqref{ode} is quasi-dissipative and maximal on $\hilbert$, it follows from the Crandall-Liggett theorem \cite[Theorem 5.6]{ik02} or \cite[Theorem 4.1 and Theorem 4.1a]{showalter2013monotone} that for any $\sol_0\in \hilbert$, resp. $\sol_0\in D(f)$, there exists a unique generalized solution $\sol\in \test([0,{1}],\hilbert)$, resp. unique strong solution $\sol\in \test([0,{1}],\hilbert)$ to \eqref{ode}. Furthermore, the operator $f$ generates a (nonlinear) strongly continuous semigroup denoted by $(\mathbb{S}(t))_{t\geq 0}$, meaning that the solution of \eqref{ode} can be written as $$\sol(t):=\mathbb{S}(t)\sol_0$$ for all $t\in [0,1]$, for all $\sol_0\in\hilbert$ for generalized solutions, and for all $\sol_0\in D(f)$ for strong solutions. In particular, if $f$ is quasi-dissipative and maximal, solutions exploding in finite-time {cannot exist}, meaning that the solution is always bounded on a finite interval of time. Therefore, there always exists a bounded solution set $\solset$ which is positively invariant.

\section{Examples}
\label{sec_examples}

\subsection{Semilinear equations}\label{secsemilinear}
%\nicolas{Maybe we could have a subsection 'Examples' whose examples would be subsubsections}

Let
\begin{equation}\label{semilinear}
f(\sol) := A\sol + g(\sol)
\end{equation}
with $g:\hilbert\rightarrow \hilbert$ a bounded and quasi-dissipative operator and $A:D(A)\subset \hilbert$ a linear operator generating a strongly continuous semigroup denoted by $(\mathbb{T}(t))_{t\geq 0}$. A function $\sol\in \test([0,1],\hilbert)$ is called a mild solution {to evolution equation \eqref{ode} with semilinear operator \eqref{semilinear} if it satisfies the Duhamel formula}
\begin{equation}
\label{duhamel}
\sol(t)= \mathbb T(t) \sol_0 + \int_0^t \mathbb T(t-s) g(\sol(s))ds,\: \sol_0\in \hilbert.
\end{equation}
In the semilinear case, generalized solutions and mild solutions coincide, as proved in \cite[Lemma 4.3.4]{vanspranghe2022}. This notion of mild solution can be defined for semilinear operators as in \eqref{semilinear}, but it is difficult to apply it for quasi-linear dynamical systems (such as conservation laws) since operator $f$ does not admit a linear part, which might therefore prevent the use of the Duhamel formula.

Moreover, solving the implicit formulation \eqref{duhamel}, when it is possible, gives the semigroup $(\mathbb{S}(t))_{t\geq 0}$, generated by the operator $f$. In other words, the solution of \eqref{duhamel} can be written as $\sol(t):=\mathbb{S}(t)\sol_0$ for a.e. $t\geq 0$.

\subsection{The heat equation}

Let $\hilbert:=L^2(\Omega)$ with $\Omega$ a bounded open set of $\R^n$ of class $C^{\infty}$. Let
\begin{equation}\label{heatop}
f(\sol):=\Delta\sol
\end{equation}
where $\Delta := \sum_{i=1}^n \partial_{x_i x_i}$ is the Laplacian operator.
The domain is
\[
D(f):=H^2(\Omega)\cap H^1_0(\Omega).
\]

\begin{lemma}\label{lapqm}
	The heat operator \eqref{heatop} is quasi-dissipative on $\hilbert$.
\end{lemma}

{\bf Proof:}
Since $f$ in  \eqref{heatop} is linear, the quasi-dissipative condition \eqref{monotone} on $\hilbert$ becomes
\[
\scal{\sol}{f(\sol)} \leq a |\sol|^2
\]
for some $a\geq 0$ and for all $\sol \in \hilbert$. It is satisfied for $a=0$ since
\[
\scal{\sol}{f(\sol)} = \int_{\Omega} (\Delta\sol)\sol = -\int_{\Omega}|\partial_x \sol|^2 \leq 0
\]
with $\partial_x := (\partial_{x_i})_{i=1,\ldots,n}$ denoting the gradient operator.
$\Box$

\begin{lemma}\label{lapmax}
	The heat operator \eqref{heatop} is maximal on $\hilbert$.
\end{lemma}

{\bf Proof:} See \cite[Thm. 9.25]{b11}. $\Box$

\subsection{Reaction-diffusion}

Let $\hilbert:=L^2(0,1)$ and consider the following periodic reaction-diffusion operator
\begin{equation}\label{prdop}
f(\sol) := \partial_{xx} \sol + g(\sol)
\end{equation}
with domain 
\[
D(f) := \lbrace \sol \in H^2(0,1) : \sol[0] = \sol[1], \partial_x \sol[0] = \partial_x \sol[1] \rbrace
\]

and Fr\'echet differentiable operator $g$ from $\hilbert$ to $\hilbert$ (the operator $g$ does not have to be polynomial) such that $g(\sol)\didier{[x]} \leq 0$ if $\sol\didier{[x]} \leq \sol_{\min}$ or $\sol\didier{[x]} \geq \sol_{\max}$ for given bounds $\sol_{\min} \leq \sol_{\max}$.
\didier{For $y \in H^2(0,1)$, the notation $y[x]$ refers to the evaluation of function $y$ at the point $x \in [0,1]$.}

%\didier{A formal definition of a polynomial operator $g$ from $\hilbert$ to $\hilbert$ is defined as a finite linear combination of finite products of linear functionals: \[ g(y) := \sum_{i=1}^{n_g} g_i \prod_{j=1}^{n_i} \ell_{i,j}(y) \] for given coefficients $g_i \in \hilbert$, and given linear functionals $\ell_{i,j}$ from $\hilbert$ to $\R$, with $j=1,\ldots,n_i$, $i=1,\ldots,n_g$.}

\begin{lemma}\label{prdsolset}
	The set $\solset := \{y \in D(f) : y_{\min} \leq y\didier{[x]} \leq y_{\max}, \:\forall x\in(0,1) \}$ is positively invariant for evolution equation \eqref{ode} with operator \eqref{prdop}. 
\end{lemma}

\begin{lemma}\label{prdqm}
	The reaction-diffusion operator \eqref{prdop} is quasi-dissipative on the positively invariant set $\solset$ of Lemma \ref{prdsolset}.
\end{lemma}

\begin{lemma}\label{prdmax}
	The reaction-diffusion operator \eqref{prdop} is  maximal on the positively invariant set $\solset$ of Lemma \ref{prdsolset}.
\end{lemma}

{The proofs of these results are provided in the Appendix.}

\section{Linear measure formulation and notions of solutions}\label{sec:linear}

In this section we introduce the Liouville equation\footnote{ {The Liouville equation is also called the continuity equation or the conservation of mass equation. It is a linear transport equation appearing e.g. in the Euler equations of fluid dynamics.}}
in an infinite dimensional separable and real Hilbert space $\hilbert$ in order to transform the nonlinear  evolution equation \eqref{ode} into a linear evolution equation with measures as unknowns. Then, we aim at proving the absence of relaxation gap between the nonlinear equation and its linear measure reformulation. No relaxation gap is instrumental to prove convergence of a numerical scheme, known as the moment-SOS hierarchy, based on well-known tools from {convex optimization}. This Liouville equation, as we will see, has to be considered in a weak sense, and this requires the introduction of  {cylindrical} test functions. In order to make this article as self-contained as possible, we recall their definition \cite[Definition 5.1.11]{ags08}.

\begin{definition}[Cylindrical function]
	Given an integer $d$, we denote by $\Pi_d(\hilbert)$ the space composed by all projective maps $\pi: \hilbert \rightarrow \mathbb{R}^d$ of the form:
	$$
	%\pi(\sol) = (\langle \sol,e_1\rangle,\langle \sol,e_2\rangle,\ldots,\langle \sol,e_d\rangle), \: \sol \in \hilbert,
	\pi(z) = (\langle z,e_1\rangle,\langle z,e_2\rangle,\ldots,\langle z,e_d\rangle), \: z \in \hilbert
	$$
	where $\lbrace e_1,\ldots, e_d\rbrace$ is any orthonormal family of vectors in $\hilbert$. We denote by $\testcyl(\hilbert)$ the space of cylindrical functions $\phi$ defined as $\phi:= \psi\circ \pi$ with $\pi\in\Pi_d(\hilbert)$ and $\psi \in \test_c^\infty(\mathbb{R}^d)$, smooth functions with compact support.
\end{definition}
By definition, cylindrical functions are Lipschitz continuous and everywhere Fr\'echet differentiable with respect to the weak topology of $\hilbert$. In the sequel, the cylindrical test functions depend on time, i.e. $\phi:=\psi(t,\pi(z))$, with $\psi \in \test_c^\infty([0,1]\times \mathbb{R}^d)${, which actually is equivalent to take $\phi \in \testcyl([0,{1}] \times \hilbert)$.}
{Cylindric  functions are necessary from a theoretical viewpoint, since our proofs are based on results of  \cite{ags08} relying on them. They are also crucial when developing the numerics because they can model polynomials in infinite dimensions.
	
	Following \cite{hikv23}, consider the time-dependent Dirac measure %$\mu(t,.)=\delta_{\sol(t,.)}$ 
	$\mu_t=\delta_{y(t)}$
	supported on the strong solution $\sol(t)$ to evolution equation \eqref{ode} for a given initial condition $\sol_0 \in \hilbert$.
	Let $\phi \in \testcyl([0,{1}] \times \hilbert)$ be a cylindrical test function. It holds
	\[
	%\int_0^T\dot{\phi}(t,\sol(t,.))dt = 
	%\phi(T,\sol(T,.))-\phi(0,\sol(0,.)) = \int_{\hilbert} \phi(T,\sol) \mu(T,d\sol) - \int_{\hilbert} \phi(0,\sol)\mu(0,d\sol),
	\int_0^{{1}}\dot{\phi}(t,\sol(t))dt = 
	\phi(1,\sol({1}))-\phi(0,\sol(0)) = \int_{\hilbert} \phi({1},z) d\mu_{{1}}(z) - \int_{\hilbert} \phi(0,z)d\mu_0(z),
	\]
	and using the chain rule
	\[
	\begin{split}
	\int_0^{{1}} \dot{\phi}(t,\sol(t))dt & = 
	\int_0^{{1}} \left(\partial_t \phi(t,\sol(t)) + \partial_{\sol} \phi(t,\sol(t))\dot{\sol}(t)\right) dt \\
	& = 
	\int_0^{{1}}  \int_{\hilbert} \left( \partial_t \phi(t,z) +  \langle \partial_{z} \phi(t,z), f(z)\rangle_{\hilbert} \right) d\mu_t(z)
	\end{split}
	\]
	where $\partial_t \phi(t,z)$ is the partial derivative of $\phi$ w.r.t. time, and $\partial_{z} \phi(t,z)$ is the Fr\'echet derivative of $\phi$ w.r.t. $z \in \hilbert$, a linear operator on $\hilbert$, which exists due to the definition of cylindrical functions.
	
	Equating both expressions, the Dirac measure $\mu_t = \delta_{y(t)}$ solves 
	the Liouville equation:
	\begin{equation}
	\label{testliouville}
	\int_0^{{1}} \int_\hilbert \left(
	\partial_t \phi(t,z) + \langle \partial_{z} \phi(t,z),f(z)\rangle\right) d\mu_t(z)dt = \int_\hilbert \phi({1},z) d\mu_{{1}}(z) - \int_\hilbert \phi(0,z) d\mu_0(z)
	\end{equation}
	which is a linear transport equation in $\mu : [0,{1}] \to \meas(\hilbert)$, where $\meas(\hilbert)$ denotes the vector  space of measures\footnote{{In this paper by measure we mean a positive Radon measure, i.e. locally finite and tight.}} on $\hilbert$, identified with bounded linear functionals on the vector space $\test(\hilbert)$ of continuous functions on $\hilbert$. More generally, let $\mu_t$ denote a solution of \eqref{testliouville}.
	{Time-dependent measures $\mu_t$ are called parametrized measures or Young measures in the calculus of variations literature.}
	
	Let $\nu \mapsto -\partial_{\sol} \cdot (f \nu)$ denote the linear operator which is adjoint to the linear operator
	$\psi \mapsto \partial_{\sol} \psi f(\sol)$, i.e. such that
	for every cylindrical function $\psi \in \testcyl(\hilbert)$ and every measure $\nu \in \meas(\hilbert)$, it holds
	$\int_{\hilbert} \partial_y \psi(z)f(z)d\nu(z) = - \int_{\hilbert} \psi(z) \: \partial_{z} \cdot (f\nu)(dz)$.
	The Liouville equation \eqref{testliouville} can then be equivalently written as an evolution equation on measures:
	\begin{equation}\label{liouville}
	\dot{\mu_t} + \partial_{z} \cdot (f\mu_t) = 0
	\end{equation}
	with $\mu_0=\delta_{y_0}$, for given initial data ${\sol_0}\in\hilbert$.
	
	Consider the semigroup $(\mathbb{S}(t))_{t\geq 0}$ generated by the operator $f$. Trajectories of  \eqref{ode} can be therefore written as $y(t,\cdot)=\mathbb{S}(t)y_0$ with $y_0 \in D(f)$ and for all $t\geq 0$. Observe that if the initial data is an arbitrary probability measure $\mu_0 \in \meas(\hilbert_1)$, then the push-forward measure $\mu_t = \mathbb{S}(t)_\# \mu_0$ through the flow map solves Liouville equation \eqref{liouville}.
	Indeed, in this case \eqref{testliouville} writes
	\[
	\begin{split}
	\int_0^{{1}} \int_\hilbert & \left(
	\partial_t \phi(t,z) + \langle \partial_z \phi(t,z), f(z)\rangle\right) d\mu_t(z)dt\\
	& = {
		\int_0^{{1}} \int_\hilbert \left(
		\partial_t \phi(t,\mathbb S(t)z) + \langle\partial_{z} \phi(t,\mathbb S(t)z),f(\mathbb S(t)z)\rangle\right) d\mu_0(z)dt }\\
	& = {
		\int_0^{{1}} \int_\hilbert \left(
		\partial_t \phi(t,\mathbb S(t)z) + \langle\partial_z \phi(t,\mathbb S(t)z),\dot{\bigl(\mathbb{S}(t)z\bigr)}\rangle\right) d\mu_0(z)dt }\\
	& = {
		\int_0^{{1}} \int_\hilbert \dot{\phi}(t,\mathbb S(t)z) d\mu_0(z)dt }\\
	& = {
		\int_{\hilbert} \phi({{1}},\mathbb S(1)z)d\mu_0(z) -\int_{\hilbert} \phi(0,\mathbb S(0)z)d\mu_0(z) }\\
	& = 
	\int_{\hilbert} \phi({{1}},z)\mu_{{1}}(dz) -\int_{\hilbert} \phi(0,z)\mu_0(dz).
	\end{split}
	\]
	It is worth mentioning that the (formal) computations performed all along this section require the solution to \eqref{ode} to be strong. Otherwise, the operator $f(\mathbb{S}(t)z)$, for all $t\geq 0$, would not exist. In other words, the weak formulation \eqref{testliouville} with test functions actually corresponds to the strong solution to \eqref{ode}. Therefore, as in the case of \eqref{ode}, we need different notions of solution. 
	\begin{definition}[Strong measure-valued solution]
		\label{def:strongmvsol}
		Given $\mu_0\in \meas(\hilbert_1)$, every strongly narrowly continuous measure curve\footnote{{A measure curve $\mu_t$ is strongly narrowly continuous if the real-valued map $t\in [0,1]\mapsto \int_\hilbert \phi(z)d\mu_t(z)\in\mathbb R$ is continuous for every continuous  bounded function $\phi$.}}  $\mu_t$ solving \eqref{testliouville} is called a strong measure-valued solution {of the Liouville equation \eqref{liouville} associated with evolution equation \eqref{ode}.} 
		%\nicolas{Moreover, we should say a few words about the 'strongly narrowly continuous,' which seems to have come out of nowhere here.}
	\end{definition}
	
	\begin{remark}
		\label{rem:bounded}
		Note that the measure $\mu_t=\mathbb{S}(t)_\#\mu_0$ with $\mu_0\in \meas(\hilbert_1)$ is a strong measure-valued solution. In other words, existence of solution is straightforward. However, uniqueness cannot be deduced easily. 
		
		To prove uniqueness, we will use the Wasserstein distance, and this requires second order moments to be bounded. 
		Using the notion of strongly narrowly continuous measure curve, we can deduce that the second order moments are bounded as follows.
		We first observe that $z\mapsto |z|$ is continuous and bounded on $\hilbert$, by definition of the norm. It should then follow that if $\mu_t$ is a strongly narrowly continuous measure curve, then $t\in[0,1]\mapsto\int_{\hilbert}|z|^2d\mu_t(z)$ is bounded. A similar conclusion holds for any moment $\int_{\hilbert}|z|^pd\mu_t(z)$, $p\in\N$.
	\end{remark}

	\begin{definition}[Generalized measure-valued solution]
		Given $\mu_0 \in\meas(\hilbert)$ and a sequence $(\mu_0^n)_{n\in\N} \subset \meas(\hilbert_1)$, a generalized measure-valued solution to \eqref{liouville} is a strongly narrowly continuous measure curve $\mu_t\in \meas(\hilbert)$ for which there exists a sequence of strong measure-valued solutions $(\mu_t^n)_{n\in\mathbb{N}}\subset \meas(\hilbert_1)$ with $\mu_t^n$ converging narrowly\footnote{A sequence of measures $(\mu^n)_{n \in \N}$ converges narrowly to a measure $\mu$ if $(\int_\hilbert \phi(z)d\mu^n(z))_{n \in \N}$ converges to $\int_\hilbert \phi(z)d\mu(z)$ for every continuous  bounded function $\phi$.}
		to $\mu_t$ 
		%in the weak star topology induced by $\testcyl([0,1]\times \hilbert)$ 
		when $\mu_0^n$ converges narrowly to $\mu_0$.
		%in the weak star topology induced by $\testcyl(\hilbert)$.
	\end{definition}
	
	\begin{remark}[The semilinear case]
		Again, by specifying the semilinear case as in Section \ref{secsemilinear}, one can reformulate \eqref{liouville} as follows:
		\begin{equation}\label{testliouvillesemi}
		\begin{split}
		\int_0^1 \int_\hilbert &\left(
		\partial_t \phi(t,z) + \langle A^\ast\partial_z \phi(t,z),z\rangle + \langle \partial_z \phi(t,z),g(z)\rangle \right) \mu_t(dz)dt = \\
		&\int_\hilbert \phi(1,z) \mu_1(dz) - \int_\hilbert \phi(0,z) \mu_0(dz)
		\end{split}
		\end{equation}
		{where $A^\ast$ is the adjoint operator of $A$.} Taking cylindrical test function $\phi\in \testcyl([0,T]\times D(A^\ast))$ and recalling furthermore that $D(A^\ast)$ is a dense subset of $\hilbert$ since $A$ generates a strongly continuous semigroup, the latter equation makes even sense when considering generalized solutions. 
	\end{remark}
	
	\section{Main result: no relaxation gap}\label{sec:nogap}
	
	%\nicolas{parler des solutions Dirac + du push-forward $y_0$ + moments bien définis}
	
	In this section, we state and prove our main results: first, we show that strong (and generalized) solutions to the Liouville equation are uniquely defined; second, as a corollary, it can be deduced that, for any initial measure $\mu_0$ equal to the Dirac measure $\delta_{y_0}$ on $\hilbert_1$ (resp. $\hilbert$), the solution $\mu_t$ to the Liouville equation which is the Dirac measure $\delta_{y(t|y_0)}$ on $\hilbert_1$ (resp. $\hilbert$) corresponding to the strong solution to \eqref{ode} (resp. the generalized solution to \eqref{ode}). Note that the existence of (either strong or generalized) solutions to the Liouville equation can be proved easily if one sets $\mu_t=\mathbb{S}(t)_\#\mu_0$ with $\mu_0\in \meas(\hilbert)$, meaning that the well-posedness of \eqref{testliouville} reduces to proving the uniqueness.
	Our strategy relies on the use of the Wasserstein distance, which requires the second order moment of the (either strong or generalized) measure-valued solution to be bounded. Since the strong measure-valued solutions are strongly narrowly continuous, one can easily deduce that not only the second order moment is bounded, but also any higher order moments -- recall Remark \ref{rem:bounded} -- allowing therefore the use of the Wasserstein distance. %\didier{For this claim, we should here refer to a technical remark stated before Definition 8.}

	%\saroj{We consider measure-valued curves $\mu_t:[0,1] \to P^2(\hilbert)$  that are the absolutely continuous curves in the set $P^2(\hilbert)$ of probability measures on $\hilbert$ with finite second order moment\cite[Def. 1.1.1]{ags08}.}. An example such curves are the solutions to \eqref{testliouville} of the form $\mu_t = \delta_{\sol(t,\sol_0)}$, the Dirac measure concentrated at the solution $\sol(t,\sol_0)$ of \eqref{ode} (in this case $\mu(t,\cdot)$ will have finite second-order moment). \swann{Swann : Is this an additional assumption ?}) 
	%\nicolas{We are adding the restriction $\mu_t\in P^2(\hilbert)$, so we should say something about the compatibility with Definition~\ref{def:strongmvsol}.}
	%It turns out that this is the only solution of  \eqref{liouville} if the initial measure is the Dirac measure at the initial data of \eqref{ode}. In other words, there is no relaxation gap when reformulating nonlinear evolution equation \eqref{ode} as a linear transport equation \eqref{liouville}:
	%\swann{Swann: I suggest therefore to use the notation $\mathbb S(t) y_0$ (or the flow map that you defined earlier)}
	
	%\nicolas{Here, this lemma seems to come from nowhere, I suggest to move thos lemma inside the proof of thm 1 when needed}
	
	\begin{theorem}\label{thm:uniq}
		Suppose that operator $f$ is quasi-dissipative and maximal. Consider two initial measures $\check\mu_0$ and $\hat\mu_0 \in \meas(\hilbert)$, and any associated generalized measure-valued solutions $\check\mu_t$ and $\hat\mu_t$. Then, for any $t\in[0,1]$, we have
		\begin{equation}
		W_2(\check\mu_t,\hat\mu_t) \leq W_2(\check\mu_0,\hat\mu_0) e^{at}.
		\end{equation}
	\end{theorem}
	
	\begin{proof}
		The proof is divided in two parts. In the first part we consider $\check\mu_0$ resp. $\hat\mu_0 \in \meas(\hilbert_1)$ and the associated strong measure-valued solutions $\check\mu_t$ resp. $\hat\mu_t$ to \eqref{testliouville}.
		Let us consider $V(t)=W_2^2(\check\mu_t,\hat\mu_t)$. Its time derivative is given by
		\[
		\frac{d}{dt} V(t) = \lim_{h\to0} \frac{1}{h}\bigl(W_2^2(\check\mu_{t+h},\hat\mu_{t+h}) - W_2^2(\check\mu_{t},\hat\mu_{t}) \bigr) .
		\]
		where $W_2(\mu_1,\mu_2)$ is the Wasserstein distance between two probability measures $\mu_1,\mu_2 \in P(\hilbert)$, i.e.
		\[
		W^2_2(\mu_1,\mu_2):=\min_{\gamma \in \Gamma(\mu_1,\mu_2)} \int_{\hilbert^2} |z_1-z_2|^2 d\gamma(z_1,z_2)
		\]
		where
		\[
		\Gamma(\mu_1,\mu_2) := \{\gamma \in P(\hilbert^2) : \pi^{z_1}_{\#}\gamma = \mu_1, \: \pi^{z_2}_{\#}\gamma = \mu_2\}
		\]
		is the set of transport plans between $\mu_1$ and $\mu_2$.
		The notation  $g_{\#}\gamma$ stands for the push-forward measure of $\gamma$ through a map $g$. The map $\pi^{z_1} : \hilbert^2 \to \hilbert, (z_1,z_2) \mapsto z_1$ is the projection on the $z_1$ coordinate, so that $\pi^{z_1}_{\#}\gamma$ is the $z_1$ marginal of $\gamma$, and similarly for $\pi^{z_2}$. Let us also denote by
		\[
		\Gamma^*(\mu_1,\mu_2):=\{\gamma \in \Gamma(\mu_1,\mu_2) : W^2_2(\mu_1,\mu_2)=\int_{\hilbert^2} |z_1-z_2|^2 d\gamma(z_1,z_2)\}
		\]
		the set of optimal transport plans.
		
		%	Using \cite[Prop. 8.4.6]{ags08}, for absolutely continuous curves in Wasserstein space, the following property holds,\saroj{
		%\begin{gather}\label{eq:first_order_approx}	
		% W_2(\mu_{t+h},(I-hf)_{\#}\mu_t) = o(h).
		% \end{gather}

		For some small real $h$, we introduce the map $F_h\colon z\mapsto z+hf(z)$, with $z\in\hilbert_1$. Using Lemma \ref{lemma:wasserstein_distance_approx} in the Appendix, the following property holds:
		\begin{gather}
		W_2(\mu_{t+h},{F_h}_{\#}\mu_t) = o(h).
		\end{gather}
		As a consequence, it is sufficient to consider the limit of
		\[
		\frac{1}{h}\bigl(W_2^2({F_h}_{\#}\check\mu_t,{F_h}_{\#}\hat\mu_t)) - W_2^2(\check\mu_{t},\hat\mu_{t}) \bigr)
		\]
		when $h\to0$ in order to deduce the time derivative of $V$. Let us now use first order differential calculus, as in the proof of \cite[Thm. 8.4.7]{ags08}. Select $\gamma \in \Gamma^*(\check\mu_t,\hat\mu_t)$, and define
		\[
		\gamma_h:= \left(F_h\circ\pi^{z_1}, F_h\circ \pi^{z_2}\right)_{\#}\gamma
		\]
		where $\circ$ is the composition operator, and note that $\gamma_h\in\Gamma({F_h}_{\#}\check\mu_t,{F_h}_{\#}\hat\mu_t)$. By definition of the Wasserstein distance:
		\begin{gather}\label{eq:inequality}
		W_2^2({F_h}_{\#}\check\mu_t,{F_h}_{\#}\hat\mu_t)) \leq \int_{\hilbert^2} |z_1-z_2|^2 d\gamma_h(z_1,z_2) .
		\end{gather}
		On the other hand,
		\begin{align*}
		\int_{\hilbert^2} |z_1-z_2|^2 d\gamma_h(z_1,z_2) &= \int_{\hilbert^2} |z_1+hf(z_1)-z_2-hf(z_2)|^2 d\gamma(z_1,z_2) \\
		&= \int_{\hilbert^2} |z_1-z_2|^2 d\gamma(z_1,z_2) \\
		&\quad + 2h \int_{\hilbert^2} \scal{z_1-z_2}{f(z_1)-f(z_2)} d\gamma(z_1,z_2) + O(h^2) \\
		&= W_2^2(\check\mu_t,\hat\mu_t) + 2h \int_{\hilbert^2} \scal{z_1-z_2}{f(z_1)-f(z_2)} d\gamma(z_1,z_2) + O(h^2) .
		\end{align*}
		Combining this equality with in the previous inequality \eqref{eq:inequality}, and letting $h\to0^-$ and $h\to0^+$ successively, one obtains
		\[
		\frac{d}{dt} V(t) = 2 \int_{\hilbert^2} \scal{z_1-z_2}{f(z_1)-f(z_2)} d\gamma(z_1,z_2) .
		\]
		Using the quasi-dissipativity of operator $f$ and the Gr\"onwall lemma we have 
		\begin{gather}\label{eq:thm_strong_soln}
		W_2(\check\mu_t,\hat\mu_t) \leq W_2(\check\mu_0,\hat\mu_0) e^{at} .
		\end{gather}
		For the second part of the proof, let $\check\mu_0$ resp. $\hat\mu_0 \in \meas(\hilbert)$ %and that, for every $n\in\mathbb N$, $\lim_{n\to \infty}\check\mu^n_0=\check\mu_0$ and similarly for $\hat\mu_0$ (w.r.t. to the narrow convergence).
		{with its corresponding narrowly converging sequence $(\check{\mu}^n_0)_{n\in \N}$ resp.
			$(\hat{\mu}^n_0)_{n\in \N} \subset \meas(\hilbert_1)$.} Let us consider the associated generalized measured-valued solution $\check\mu_t$ resp. $\hat\mu_t \in  \meas(\hilbert)$ to \eqref{testliouville}
		{with its corresponding narrowly converging sequence of strong solutions $(\check\mu^n)_{n\in \mathbb{N}}$ resp. $(\hat\mu^n)_{n\in \mathbb{N}}\subset \meas(\hilbert_1)$.} %satisfying 
		%\begin{gather}\label{eq:narrow_convergence}
		%\lim_{n\rightarrow + \infty} \int_{[0,1]\times \hilbert} \phi(t,z) d\check\mu_t^n(z) dt = \int_{[0,1]\times \hilbert} \phi(t,z) d\check\mu_t(z)dt,\: \forall \phi\in\testcyl([0,1]\times \hilbert_1)
		%\end{gather}
		% and similarly for $\hat\mu_t$. 
		The narrow convergence and the convergence of the second order moments implies convergence in Wasserstein metric \cite[Rmk. 7.1.11]{ags08}.
		Therefore, one has, for every $n\in\mathbb{N}$,
		\[
		\begin{split}
		W^2_2(\check\mu_t,\hat\mu_t)\leq W_2^2 (\check\mu_t,\check\mu_t^n) + W_2^2(\check\mu_t^n,\hat\mu^n_t) + W^2_2(\hat\mu_t^n,\hat\mu_t). 
		\end{split}
		\]
		For large $n\in\mathbb{N}$ we have $W_2^2 (\check\mu_t,\check\mu_t^n) \le \epsilon /3$ and $W^2_2(\hat\mu_t^n,\hat\mu_t) \le \epsilon/3$. For the second term in the above inequality we observe that 
		\[
		W_2^2(\check\mu_t^n,\hat\mu^n_t) \le W_2(\check\mu^n_0,\hat\mu^n_0) e^{at} \le W_2(\check\mu_0,\hat\mu_0) e^{at} +\epsilon/3.
		\]
		which follows from \eqref{eq:first_order_approx} and the convergence of $\check\mu^n_0$ (resp. $\hat\mu^n_0$) to $\check\mu_0$ (resp. $\hat\mu_0$) w.r.t. Wasserstein metric. We deduce that, as $n$ goes to infinity, we have, for a.e. $t\geq 0$
		$$
		W_2(\check\mu_t,\hat\mu_t) \leq W_2(\check\mu_0,\hat\mu_0) e^{at},
		$$
		concluding the proof of the result.\end{proof}
	\begin{corollary}[{\bf No relaxation gap}]\label{nogap}
		Suppose that $f$ is quasi-dissipative and maximal. If $\mu_0=\delta_{\sol_0}$ with $y_0\in \hilbert_1$ resp. $\hilbert$ then $\mu_t=\delta_{\mathbb S(t) \sol_0}$ is the only strong resp. generalized
		measure-valued solution of \eqref{testliouville} associated with $\mu_0$.
	\end{corollary}

	\begin{proof} Let $\mu_0 = \delta_{\sol_0}$ with $\sol_0 \in \hilbert_1$ then $\delta_{\mathbb S(t)\sol_0}$ is a solution to \eqref{testliouville}. We assume there exists another solution $\mu_t$ to \eqref{testliouville} with the same initial condition $\mu_0 = \delta_{\sol_0}$. Let us consider the time-dependent function $V(t):=W^2_2(\mu_t,\delta_{\sol(t)}) : [0,{1}] \to \R$. Then using the results from Theorem \ref{thm:uniq} we get
		$W_2(\mu_t, \delta_{\sol(t)}) = 0$
		for all time $t \in [0,1]$ as $\frac{d V(t)}{dt} \le e^{at} V(0)$ and $V(0)=0$.
		Similar result holds for $\mu_0 = \delta_{\sol_0}$ with $\sol_0 \in \hilbert$.
	\end{proof}

	\section{Infinite-dimensional moment-SOS hierarchy}\label{sec:hierarchy}
	
	\subsection{Polynomials and moments}
	\label{polynomials}
	
	Polynomials can be expressed as linear combinations of monomials. In the infinite-dimensional setting, polynomials are called
	Wiener polynomials or chaos polynomials \cite{w38}, because they have been originally introduced for stochastic differential equations. Let $c_0(\mathbb{N})$ the set of integers with finitely many non-zero elements, i.e. if $a=(a_1,a_2,\ldots)\in c_0(\mathbb{N})$, then $\mathrm{card}\lbrace i\in\mathbb{N}\mid a_i\neq 0\rbrace<\infty$. In this setting, and considering $z\in\hilbert$, a \textit{monomial} of degree $a\in c_0(\mathbb{N})$ is defined as
	$$
	z^a := \prod_{i=1}^\infty \langle z,\phi_i\rangle^{a_i}.
	$$
	This is a product of finitely many powers of linear functionals $\langle z,\phi_i\rangle$ with $i = 1,2,\ldots$ and given functions $\phi_i$ in $\hilbert'$, the topological dual of $\hilbert$.
	Polynomials in $\hilbert$ are then defined as linear combinations of monomials, i.e.
	\[
	\begin{split}
	p:\: \hilbert &\rightarrow \mathbb{R}\\
	z &\mapsto \sum_{a\in \mathrm{spt}(p)} p_a z^a
	\end{split}
	\]
	where the sum runs over $\mathrm{spt}(p)$, the support of $p$, a (possibly infinite) countable subset of $c_0(\N)$.
	Let $\R[z]$ denote the ring of polynomials.
	These polynomials have two types of degrees. To follow the terminology introduced in \cite{hikv23,hr24}, the \textit{algebraic degree} is defined as
	\[
	d:=\max_{a\in\mathrm{spt}(p)}\sum_{i=1}^\infty a_i. 
	\]
	which corresponds to the total degree in the finite-dimensional setting. The second notion of degree, namely the \textit{harmonic degree}, is defined as 
	\[
	n:=\max_{a\in \mathrm{spt}(p)} \lbrace i\in\mathbb{N}\mid a_i\neq 0\rbrace,  
	\]
	which corresponds to the number of variables in the finite-dimensional setting.
	
	Given a measure $\nu$ on $\hilbert$ and an index $a \in c_0(\N)$ the quantity
	\begin{equation}\label{moment}
	m_a := \int_\hilbert z^a d\nu(z)
	\end{equation}
	is called the \textit{moment} of order $a$ of measure $\nu$.
	
	Given a sequence $m:=(m_a)_{a \in c_0(\N)}$, let us define the Riesz functional
	\[
	\begin{split}
	\ell_m :\: \R[z] &\rightarrow \R\\
	p &\mapsto \sum_{a\in \mathrm{spt}(p)} p_a m_a.
	\end{split}
	\]
	If the sequence $m$ has a representing measure $\nu$, i.e. if \eqref{moment} holds for $a \in c_0(\N)$, then
	\[
	\ell_m(p) = \int_\hilbert p(z)d\nu(z).
	\]
	
	\subsection{Moment and localizing matrices}\label{momloc}
	
	Let us derive conditions satisfied by the moments of a measure $\mu$ supported on a subset $\solset$ of $\hilbert$. Let  $\solset := \{z \in \hilbert : p(z) \geq 0\}$ be defined as the compact superlevel set of a given polynomial $p$, see \cite{h24} for examples.
	
	Since $\mu$ is positive, the Riesz functional corresponding to the sequence $m$ of moments of $\mu$ must be positive on squares, i.e. $\ell_m(q^2_0) \geq 0$ for all $q_0 \in  \R[z].$ It must also be positive on $\solset$, i.e. $\ell_m(pq^2_1) \geq 0$ for all $q_1 \in  \R[z].$ It turns out that these necessary conditions are also sufficient for sequence $m$ to have a representing measure, this is the dual moment formulation of Jacobi's Positivstellensatz that can be found in \cite[Theorem 3.9]{ikkm23b}. See \cite[Theorem 2.1]{gkm14} for a
	reformulation of Jacobi's Positivstellensatz -- whose original statement is \cite[Theorem 4]{j01} -- which is valid for sums of squares (SOS) representations of positive polynomials in an Archimedean quadratic module of any unital commutative algebra, and in particular for the algebra generated by elements of $\R[z]$.
	
	Numerically, the sequence $m$ must be truncated up to a given algebraic and harmonic degree, i.e. the above positivity conditions are enforced for bounded degree polynomials. The positivity condition $\ell_m(q^2_0) \geq 0$ resp. $\ell_m(pq^2_1) \geq 0$ for bounded degree $q_0$ resp. $q_1$ is formulated as positive semidefiniteness of a symmetric matrix depending linearly on $m$, the so-called \textit{moment matrix} resp. \textit{localizing matrix}. Positive semidefiniteness of the moment and localizing matrices
	results in finite-dimensional convex linear matrix inequality (LMI) in the truncated moment sequence. These conditions are necessary for the entries of the truncated sequence to be moments of a measure on $\solset$. They are called moment relaxations, and the truncated sequence entries are called pseudo-moments.
	The LMI conditions grow in size with the truncation degree, and they become sufficient asymptotically, i.e. for infinitely many constraints. This is the essence of the infinite-dimensional moment-SOS hierarchy,
	as described in \cite[Section 5]{hikv23} and \cite{hr24}.
	See also \cite{hkl20} and \cite{n23} for the finite-dimensional moment-SOS hierarchy and its applications.
	
	\subsection{Moment formulation of the Liouville equation}
	
	The weak formulation of the Liouville equation \eqref{testliouville} becomes a linear equation in the moments of measure $\mu_t(dz)dt$, provided we use monomials of both $t$ and $z$ in the test functions
	\begin{equation}\label{test}
	\phi_a(t,z) = t^{a_0} \prod_{i=1}^{\infty} \scal{\phi_i}{z}^{a_i}
	\end{equation}
	for each given index $a=(a_0,a_1,a_2,\ldots) \in c_0(\N)$.
	Let us define resp.
	\begin{equation}\label{moments}
	\begin{array}{rcl}
	m^{0,1}_a & := & \int_0^1 \int_{\solset} \phi_a(t,z)\mu_t(dz)dt \\
	m^0_a & := & \int_{\solset} \phi_a(0,z)\mu_0(dz) \\
	m^1_a & := & \int_{\solset} \phi_a(1,z)\mu_1(dz) \\
	\end{array}
	\end{equation}
	the occupation resp. initial and terminal moments of order $a$.
	
	For each given index $a \in c_0(\N)$, the Liouville equation \eqref{testliouville} corresponds to a linear equation
	\begin{equation}\label{linequ}
	L_a(m^{0,1})+m^1_a = m^0_a.
	\end{equation}
	where $L_a$ is a given linear functional of the sequence of occupation moments $m^{0,1}=(m^{0,1}_a)_{a \in c_0(\N)}$, and $m^0:=(m^0_a)_{a \in c_0(\N)}$ resp. $m^1:=(m^1_a)_{a \in c_0(\N)}$ is the sequence of initial resp. terminal moments, as defined in \eqref{moments}.
	
	By enumerating all index sequences, we generate countably  infinitely many linear moment equations \eqref{linequ}, resulting in an infinite-dimensional linear system of equations in the moment sequences $m^{0,1}$, $m^0$ and $m^1$.
	Each equation involves infinitely many moments, so for computational purposes we have to truncate the infinite sums in the expression of linear functional $L_a$ to finitely many moments. Let us denote by $L^h_a$ the corresponding linear functional truncated to harmonic degree $h$. The remaining terms are absorbed by a error residual denoted $e^h_a$, so that linear equation \eqref{linequ} becomes
	\begin{equation}\label{linres}
	L^h_a(m^{0,1})+m^1_a=m^0_a+e^h_a(m^{0,1}).
	\end{equation}
	It is possible to get estimates of the error residual, but for computational purposes it may suffice to minimize its quadratic norm.
	
	\subsection{Moment-SOS hierarchy}
	
	The moment-SOS hierarchy then consists of minimizing the quadratic norm of the error residual subject to the linear equations \eqref{linres} for all $a \in c_0(\N)$ such that $|a|_1 \leq r$, for increasing values of $r$, and the LMI conditions on the moment and localizing matrices described in Section \ref{momloc}. Since there is a unique measure solution to the Liouville equation, and a measure on compact set is uniquely determined by its moments, the overall numerical procedure generates approximations to the moments of the solutions that converge pointwise. Note also that more general optimization problems can be formulated in the same framework e.g. optimization over initial conditions, introduction of control parameters etc. Precise statements and convergence proofs lie however outside of the scope of this paper.
	
	Finally, at a given relaxation order $h$, we can approximate the solution of the original equation by using an infinite-dimensional extension of the Christoffel-Darboux polynomial \cite{h24}, in analogy with what was achieved in the finite-dimensional case \cite{mpwhl21}. This approximation strategy, together with its convergence guarantees, are also outside of the scope of this paper.
	
	\section{Numerical results for polynomial reaction-diffusion }\label{sec:examples}
	
	Now let us follow the approach of  {Section \ref{sec:hierarchy}} and formulate the Liouville equation for a quadratic diffusion operator
	\[
	f(\sol) := \partial_{xx} \sol + \epsilon \sol(1-\sol)
	\]
	for given $\epsilon \geq 0$, with domain 
	\[
	D(f) := \lbrace \sol \in H^2(0,1) : \sol\didier{[0]} = \sol\didier{[1]}, \partial_x \sol\didier{[0]} = \partial_x \sol\didier{[1]} \rbrace
	\]
	the periodic functions with square integrable weak derivatives on $[y_{\min},y_{\max}]:=[0,1]$.
	Let $T:=1$ be the terminal time.
	
	In the weak formulation \eqref{testliouville} of the Liouville equation,  {consider monomial test functions \eqref{test} with dual functions $\phi_i \in \hilbert'$.} In our case, for periodic functions of $\hilbert=\hilbert'=L^2(0,1)$, the natural choice would be the complex exponentials $\phi_i(x)=e_{k_i}(x):=\exp(2\pi\sqrt{-1}k_ix)$ for a given $k_i \in \Z$. Test function \eqref{test} has algebraic degree $|a|_1:=\sum_{i=0}^{\infty} a_i$ and harmonic degree $|k|_{\infty}:=\max_{i=1}^{\infty} |k_i|$.
	
	Define the Fourier transform $F:L^2(0,1) \to \ell_2(\Z), \:\: z \mapsto c=(c_k)_{k\in\Z}$ where $c_k:=\scal{e_k}{z}$ is the $k$-th Fourier coefficient of $z \in \hilbert$.
	The adjoint of $F$ is the inverse Fourier transform $F^*:\ell_2(\Z) \to L^2(0,1), \:\: c \mapsto z = \sum_{a\in\Z} c_a e_a$. 
	Given a measure $\mu$ on $\hilbert$, let $\nu := F_{\#}\mu$ denote its push-forward measure through $F$, so that for our choice of monomials test functions \eqref{test} it holds
	\[
	\int_{0}^1 \int_{\mathscr H} \phi(t,z)\mu_t(dz)dt = \int_0^1 \int_{\ell^2} t^{a_0} c^{a_1}_{k_1} \ldots c^{a_d}_{k_d} \nu_t(dc)dt
	\]
	i.e. moments of $\mu$ become standard algebraic moments of $\nu$.  {Consistently with \eqref{moments},} let us define resp.
	\[
	\begin{array}{rcl}
	m^{0,1}_a & := & \int_0^1 \int_{\ell^2} t^{a_0} c^{a_1}_{k_1} \ldots c^{a_d}_{k_d} \nu_t(dc)dt \\
	m^0_a & := & 0^{a_0} \int_{\ell^2} c^{a_1}_{k_1} \ldots c^{a_d}_{k_d} \nu_0(dc) \\
	m^1_a & := & 1^{a_0} \int_{\ell^2} c^{a_1}_{k_1} \ldots c^{a_d}_{k_d} \nu_1(dc)
	\end{array}
	\]
	the occupation resp. initial and terminal moments of order $a$.
	
	Now observe that reporting these test functions in linear equation \eqref{testliouville} we can express the following terms with our moments:
	\[
	\begin{array}{rclcl}
	\int_{\hilbert}\phi(0,z)\mu_0(dz) & = & 0^{a_0} \int_{\ell^2} c^{a_1}_{k_1} \ldots c^{a_d}_{k_d} \nu_0(dc) & = & m^0_a \\
	\int_{\hilbert}\phi(1,z)\mu_1(dz) & = & 1^{a_0} \int_{\ell^2} c^{a_1}_{k_1} \ldots c^{a_d}_{k_d} \nu_1(dc)  & = &  m^1_a \\
	\int_0^1 \int_{\hilbert} \partial_t \phi(t,z)\mu_t(dz)dt & = &
	a_0 \int_0^1 \int_{\ell^2} t^{a_0-1} c^{a_1}_{k_1} \ldots c^{a_d}_{k_d} \nu_t(dc)dt
	& = & a_0 m^{0,1}_{a_0-1,a_1,\ldots,a_d}.
	\end{array}
	\] 
	The Fr\'echet derivative of $\phi$ with respect to $y \in \hilbert$ at $f  \in \hilbert$ is given by
	\[
	\partial_z \phi(t,z)(f) = \sum_{j=1}^d a_j
	\frac{\scal{\phi_j}{f}}{\scal{\phi_j}{z}}
	\phi(t,z)
	\]
	i.e.
	\[
	\begin{split}
	\partial_z \phi(t,z)(f) & =  a_1 \scal{\phi_1}{f} t^{a_0} \scal{\phi_1}{z}^{a_1-1} \scal{\phi_2}{z}^{a_2} \cdots \scal{\phi_d}{z}^{a_d} \\ & +  a_2 \scal{\phi_2}{f} t^{a_0} \scal{\phi_1}{z}^{a_1} \scal{\phi_2}{z}^{a_2-1} \cdots \scal{\phi_d}{z}^{a_d} + \cdots
	\end{split}
	\]
	The Fourier coefficients of $f$ are given by
	\[
	\begin{split}
	\scal{\phi_i}{f} = \scal{e_{k_i}}{f} & =   - \scal{e_{k_i}}{\partial_{xx}z} - \scal{e_{k_i}}{z} + \scal{e_{k_i}}{z^2} \\
	& =  (2\pi k_i)^2\scal{e_{k_i}}{z} - \scal{e_{k_i}}{z} + \scal{e_{k_i}}{\sum_{k \in \Z} (\sum_{l \in \Z} c_l c_{k-l}) e_k} \\
	& = ((2\pi k_i)^2-1)c_{k_i} + \sum_{l \in \Z} c_l c_{k_i-l} \\
	\end{split}
	\]
	where the non-linear term
	\begin{equation}\label{nonlinear}
	z^2 = (\sum_{k \in \Z} c_k e_k)(\sum_{l \in \Z} c_l e_l) = \sum_{k,l \in \Z} c_k c_l e_{k+l} = \sum_{k \in \Z} (\sum_{l \in \Z} c_l c_{k-l}) e_k
	\end{equation}
	depends on infinitely many Fourier coefficients. 
	Equation \eqref{testliouville} can then be written as the linear moment equation
	\[
	\begin{split}
	& \int_{\hilbert} \phi(1,z)\mu_1(dz)-\int_{\hilbert} \phi(0,z)\mu_0(dz)
	= \int_0^1 \int_{\mathscr H}
	\left(\partial_t \phi(t,z) +  \partial_z \phi(t,z)(f)\right)\mu_t(dz) dt \\
	&  = 1^{a_0} \int_{\ell^2} c^{a_1}_1 \cdots c^{a_d}_d \nu_1(dc) - 0^{a_0} \int_{\ell^2} c^{a_1}_1 \cdots c^{a_d}_d \nu_0(dc) = 
	a_0 \int_0^1 \int_{\ell^2} t^{a_0-1} c^{a_1}_1 \ldots c^{a_d}_d \nu_t(dc)dt \\
	& + a_1 \int_0^1 \int_{\ell^2} t^{a_0} \left(((2\pi k_1)^2-1)c_{k_1} + \sum_{l \in \Z} c_l c_{k_1-l}\right) c^{a_1-1}_{k_1} c^{a_2}_{k_2} \cdots c^{a_d}_{k_d} \nu_t(dc)dt \\
	& + a_2 \int_0^1 \int_{\ell^2} t^{a_0} \left(((2\pi k_2)^2-1)c_{k_2} + \sum_{l \in \Z} c_l c_{k_2-l}\right) c^{a_1}_{k_1} c^{a_2-1}_{k_2} \cdots c^{a_d}_{k_d} \nu_t(dc)dt  + \cdots 
	\end{split}
	\]
	{which has the linear form \eqref{linequ}.}
	
	Let $\epsilon=0.1$ and $h=4$. The initial moment sequence $m^0$ is given, corresponding to a Gaussian distribution with mean $0$ and standard deviation $10^{-1/2}$. 
	
	At relaxation order $r=4$, the resulting semidefinite optimization problem has moment matrices of size 165, with moment vectors of size 3575 subject to 1100 linear equations. We solve this optimization problem with MOSEK, and we obtain approximate occupation moments $\tilde{m}^{0,1}$ and terminal moments $\tilde{m}^1$, to be compared with the occupation moments $m^{0,1}$ and terminal moments $m^1$ computed by a finite difference scheme as described in \cite{r23}. 
	
	The percentage of entries of the occupation resp. terminal moments that match within relative accuracy less than $10^{-3}$ is equal to $91\%$ resp. $93\%$.
	
	\section{Conclusion}\label{sec:conclusion}
	
	Nonlinear nonconvex optimization over PDEs can be reformulated as a linear optimization problem in a measure space, but it may happen that the optimal value on measures differ from the optimal value of the original problem: measures satisfying the transport equation may not correspond to solutions of the nonlinear PDE. In this case we say that there is a relaxation gap. In this paper we prove that this does not happen for a broad class of nonlinear PDEs, namely evolution equation on Hilbert space with a nonlinear operator satisfying quasi-monotonicity conditions.
	
	The important practical consequence of no relaxation gap is that we can guarantee the convergence of numerical approximation schemes based on an infinite-dimensional version of the moment-SOS hierarchy.
	
	Our approach is illustrated numerically on a simple reaction-diffusion equation, but our setup allows readily extensions to optimization problems over PDEs, such as approximation of the region of attraction (defined as the largest set of initial data compatible with the equation and the constraints), or optimal control \cite{t10}. %\didier{@Swann,Nicolas: on peut citer ici quelques ouvrages de références en contrôle des EDP ?}
	
	Further works in this line might be followed by considering the generator to be {\it locally quasi-dissipative}, which is a concept explained in details in \cite[Chapter 6]{ik02}. In particular, quasi-linear equations (including for instance conservation laws, the Korteweg-de Vries equation or the Kuramoto-Sivashinsky equation) can be studied through this framework, see \cite[Chapter 6.9]{ik02}. Indeed, in an earlier work \cite{mwhl20}, concentration of the measure-valued solution has been proved thanks to some contraction inequality (deduced from entropy inequalities) similar to the one obtained in the present paper and that follows from the quasi-dissipative property of the generator.
	
	Techniques from functional analysis distinct from dissipative arguments could also be used to prove the absence of a relaxation gap in the measure formulation. For example, the absence of relaxation gap for the problem of approximating the region of attraction of controlled ordinary differential equations was proved in \cite{hk14} with the help of Ambrosio's superposition principle. In \cite[Section 7.2]{at17}, the authors first recall the superposition principle in finite-dimensional Euclidean spaces (for ordinary differential equations), then in ${\mathbb R}^{\infty}$ (e.g. for stochastic differential equations) and then in abstract metric spaces (e.g., evolutionary PDE in Hilbert space or Banach space). In all these setups, it
	is fundamental to have a deeper understanding of the moment problem for measures supported in such infinite-dimensional spaces.
	First attempts along these lines are reported in \cite{iknm23,hikv23,hr24}.
	
	\section*{Acknowledgments}
	This work benefited from feedback from Giovanni Fantuzzi, David Goluskin, Salma Kuhlmann, Maria Infusino, Emmanuel Tr\'elat and Victor Vinnikov.

	\section*{Appendix}
	
	\subsection*{Proof of Lemma \ref{prdsolset}}
	
	We use Stampacchia’s truncation method as in the proof of \cite[Theorem 10.3]{b11}. Let $G : \R \to \R$ be a continuously differentiable function such that $G(s)=0$ if $s \leq 0$ and $G'(s)$ is strictly positive and bounded for $s>0$, and hence $G(s)\geq 0$ for all $s\in \R$. Let
	\[
	H(s) := \int_0^s G(\sigma)d\sigma.
	\]
	
	Let $y(t)$ denote a solution to \eqref{ode} with operator \eqref{prdop}, and define the function
	\begin{equation}\label{ymax}
	V(t):= \int_0^1 H(y(t)\didier{[x]}-y_{\max})dx.
	\end{equation}
	For $t>0$, it holds
	\[
	\begin{split}
	\dot{V}(t) & = \int_0^1 G(y(t)\didier{[x]}-y_{\max})\partial_t y(t)\didier{[x]}dx \\
	& \hspace{-2em}= \int_0^1 G(y(t)\didier{[x]}-y_{\max})\left(\partial_{xx} y(t)\didier{[x]} + g(y(t)\didier{[x]})\right)dx \\
	& \hspace{-2em}= \left[G(y(t)\didier{[x]}-y_{\max})\partial_x y(t)\didier{[x]}\right]_{x=0}^1 - \int_0^1 G'(y(t)\didier{[x]}-y_{\max})(\partial_x y(t)\didier{[x]})^2 dx \\
	& \quad + \int_0^1 G(y(t)\didier{[x]}-y_{\max})g(y(t)\didier{[x]})dx.
	\end{split}
	\]
	Since $y(t)\didier{[0]}=y(t)\didier{[1]}$ and $\partial_x y(t)\didier{[0]}=\partial_x y(t)\didier{[1]}$, the first term on the right-hand size is zero. Since $G'$ is positive and bounded, it follows that
	\[
	\dot{V}(t) \leq \int_0^1 G(y(t)\didier{[x]}-y_{\max})g(y(t)\didier{[x]})dx.
	\]
	Now observe that if $y\didier{[x]} \leq y_{\max}$ then $G(y\didier{[x]}-y_{\max})=0$, and if $y\didier{[x]} \geq y_{\max}$ then $G(y\didier{[x]}-y_{\max})\geq 0$ and $g(y)\didier{[x]}\leq 0$. Hence $\dot{V}(t) \leq 0$.
	By construction, $V(t)\geq 0$ for all $t\geq 0$.
	If $y(0)\didier{[x]} \leq y_{\max}$ for all $x \in [0,1]$ then $V(0)=0$, which implies $V(t)=0$ and hence from \eqref{ymax} it holds $y(t)\didier{[x]}\leq y_{\max}$ for all $t\geq 0$ and $x \in [0,1]$.
	
	Invariance of the lower bound follows along the same lines by replacing \eqref{ymax} with
	\[
	V(t):=\int_0^1 H(y_{\min}-y(t)\didier{[x]})dx,
	\]
	ending therefore the proof.

	\subsection*{Proof of Lemma \ref{prdqm}}
	
	It follows readily from Definition \ref{defqm} that the sum of quasi-dissipative operators is a quasi-{dissipative} operator. From Lemma \ref{lapqm}, if $g$ is quasi-dissipative then $f$ in \eqref{prdop} is quasi-dissipative. So let us prove that $g$ is quasi-dissipative on $\solset$.
	
	Given $y_1, y_2 \in \hilbert$ define the maps
	\[
	y(\tau):=y_2+\tau(y_1-y_2), \quad G(\tau) := \langle y_1-y_2, g(y(\tau)) \rangle
	\]
	for $\tau \in [0,1]$.
	Since $G$ is continuous, there exists $\bar{\tau} \in [0,1]$ such that
	\begin{equation}\label{inter}
	G(1)-G(0) = \scal{y_1-y_2}{g(y_1)-g(y_2)} = \frac{dG}{d\tau}(\bar{\tau}).
	\end{equation}
	Notice that
	\begin{equation}\label{dgdt}
	\frac{dG}{d\tau} =  \scal{y_1-y_2}{D_y g(y)\frac{dy}{d\tau}} = \scal{y_1-y_2}{D_y g(y)(y_1-y_2)}
	\end{equation}
	where
	%\[
	%\didier{$\frac{dg}{dy}$}
	%=\sum_{k=1}^d \frac{(y-r_1) \cdots (y-r_d)}{y-r_k}
	%\]
	$D_y g(y)$ is the Fr\'echet derivative of $g$ with respect to $y$. {Note that, by definition, any Fr\'echet derivative is a linear and bounded operator. Therefore, defining $$a:= \max_{y \in \solset}\left\Vert D_y g(y) \right\Vert_{\mathcal{L}(\hilbert)}$$
		where $\mathcal{L}(\hilbert)$ stands for the space of bounded operators having their domains and range equal to $\hilbert$, and $\Vert . \Vert_{\mathcal{L}(\hilbert)}$ is the associated operator norm, one has
		\begin{equation}\label{specdg}
		\scal{z}{D_y g(y) z} {\le} a|z|^2
		\end{equation}
		for all $y\in \solset$} \didier{and $z \in \hilbert$.}
	Combining \eqref{dgdt} and \eqref{specdg} and letting $z=y_1-y_2$, it follows that
	\begin{equation}\label{specdg2}
	\frac{dG}{d\tau}(\tau) \leq a|y_1-y_2|^2
	\end{equation}
	for all $y_1,y_2 \in \solset$ and all $\tau \in [0,1]$.
	Then for each given pair $y_1,y_2 \in \solset$, there is a value of $\bar{\tau} \in [0,1]$ such that \eqref{inter} holds, and quasi-dissipativity inequality \eqref{monotone} then follows by plugging \eqref{inter} into \eqref{specdg2}.
	
	\subsection*{Statement and proof of Lemma \ref{lemma:wasserstein_distance_approx}}
	
	\begin{lemma}\label{lemma:wasserstein_distance_approx}
		Suppose that $f$ is quasi-dissipative and maximal. Consider an initial measure $\mu_0 \in \meas(\hilbert_1)$, and the associated strong measure-valued solution  $\mu_t$. Then, for any $t\in(0,1)$
		\begin{gather}\label{eq:first_order_approx}	
		W_2(\mu_{t+h},{F_h}_{\#}\mu_t) = o(h)
		\end{gather}
		where $F_h\colon z\mapsto z+hf(z)$, with $z\in\hilbert_1$.
	\end{lemma}
	
	\begin{proof}
		We need to prove that $f \in \mathrm{Tan}_{\mu_t} P_2(\hilbert)$, i.e., $f\in L^2(\mu_t;\hilbert)$  belongs to the tangent bundle at $\mu_t$ for %$\mathcal{L}^1-a.e.~
		almost every $t \in [0,1]$.
		Using \cite[Prop. 8.4.5]{ags08}, this boils down to proving that 
		\begin{gather*}
		\Big(\int_{\hilbert} |f(x)|^2 d \mu_t(x)\Big)^{1/2} \le |\mu'|(t) %\mathcal{L}^1-a.e.~ 
		\text{ for almost every } t \in [0,1]
		\end{gather*}
		where $|\mu'|(t) = \lim_{s\to t}\frac{W_2(\mu_{s},\mu_t)}{|s-t|}$.
		Let $\mu_{t+h, t} \in \Gamma_o(\mu_t,\mu_{t+h})$ be the optimal transport plan between $\mu_t$ and $\mu_{t+h}$. Then
		\begin{gather*}
		\frac{|\mu_{t+h}(\phi)-\mu_t(\phi)|}{|h|} = \int_{\hilbert\times \hilbert} (\phi(x) - \phi(y)) d\mu_{t+h,t}(x)
		\end{gather*}
		$\forall \phi  \in \testcyl([0,1] \times \hilbert)$.
		We note that $\mu_{t+h, t} \to (x,x)_{\#} \mu_t$ narrowly as $h \to 0$  and by letting $t \in [0,1]$ be the point where $\mu_t$ is metrically differentiable w.r.t. to $t$ we obtain
		\begin{gather*}
		\limsup_{h\to 0}\frac{|\mu_{t+h}(\phi)-\mu_t(\phi)|}{|h|}  \le \lim_{h \to 0}\frac{W_2(\mu_{t+h},\mu_t)}{|h|} \Big(\int_{\hilbert}(\nabla  \phi(x))^2 d\mu_t(x)\Big)^{1/2} = |\mu'|(t)~||\nabla\phi||_{L^2(\mu_t,\hilbert)}
		\end{gather*}
		where we have used the H\"older inequality. Moreover 
		\begin{align*}
		\Big|\int_{0}^1 \int_{\hilbert}\partial_s \phi(t,x)d\mu_t(x) dt\Big|  &= \Big| \lim_{h\to 0} \int_{0}^1 \int_{\hilbert} \frac{\phi(t,x)- \phi(t-h,x)}{h} d\mu_t(x)dt\Big|\\  
		&\le \int \lim_{h\to 0}\frac{|\mu_{t+h}(\phi)-\mu_t(\phi)|}{|h|} dt \\
		& \le \int |\mu'|(t)~|\nabla\phi|_{L^2(\mu_t,\hilbert)} dt\\
		&\le \Big(\int_0^1 |\mu'|(t)^2 dt\Big)^{1/2}~\Big(\int_0^1 |\nabla\phi(t,x)|^2 d\mu_t(x) dt \Big)^{1/2}
		\end{align*}
		where for the first inequality we used Fatou's lemma and for the last we used Cauchy Schwarz.
		Next we use the fact that $\mu_t$ satisfies the Liouville equation \eqref{testliouville}, so \begin{align*}
		&\Big|\int_0^1 \int_{\hilbert}\nabla \phi(t,x)\cdot f(t,x) d\mu_t(x) dt\Big|\le \Big(\int_0^1 |\mu'|(t)^2 dt\Big)^{1/2}~\Big(\int_0^1 |\nabla\phi(t,x)|^2 d\mu_t(x) dt \Big)^{1/2} 
		\end{align*}
		Dividing both sides by $|\nabla\phi|_{L^2(\mu_t dt)}$  we obtain
		\[ \frac{\langle \nabla \phi, f \rangle_{L^2(\mu_t dt)}} 
		{|\nabla\phi|_{L^2(\mu_t dt)}} \le |\mu'|(t) \]
		which is the desired inequality as the left hand side is the $L^2(\mu_t dt)$ norm of $f$.
	\end{proof}
	
	\subsection*{Proof of Lemma \ref{prdmax}}
	
	{The maximality of \eqref{prdop} is equivalent to the existence of a positive value $\lambda$ such that
		$$
		\mathrm{Ran}(-f-\lambda \mathrm{I}_{\hilbert}) = \hilbert.
		$$
		Given that the inclusion $\hilbert \subset \mathrm{Ran}(-f-\lambda \mathrm{I}_{\hilbert})$ is straigtforward, it suffices to prove $\hilbert \subset \mathrm{Ran}(-f-\lambda \mathrm{I}_\hilbert)$. In other words, for $\bar \sol \in\hilbert$, one must show that there exists $\tilde \sol\in D(f)$ such that
		$$
		(\lambda \mathrm{I}_\hilbert-A)\tilde \sol = \bar \sol + g(\tilde \sol), 
		$$
		where $A=\partial_{xx}$ with $D(A)=D(f)$. Note that any $\lambda>0$ belongs to the resolvent of $A$ since the latter is maximal dissipative. To prove the maximality of $-f$, we will use a fixed-point strategy, based on the following mapping:
		\[
		\begin{split}
		\mathcal{T}:\: \hilbert &\rightarrow D(A)\\
		\sol &\mapsto (\lambda \mathrm{I}_\hilbert-A)^{-1}\left[\bar \sol + g(\sol)\right].
		\end{split}
		\]
		The operator $A$ being closed and having a compact resolvent, then, invoking \cite[Proposition 4.24]{cr21}, the injection from $D(A)$ to $\hilbert$ is compact. Therefore, the set
		$$
		\mathcal{B}:= \lbrace \sol \in D(A)\mid |\sol|_{D(A)}\leq N\rbrace,
		$$
		with $N$ a positive constant to be defined later, is compact and convex, as a ball of radius $N$ and centered at $0$. Denoting by $L$ the Lipschitz constant of $g$, one has
		\[
		\begin{split}
		|\mathcal{T}(\sol)|_{D(A)} \leq &\Vert (\lambda \mathrm{I}_\hilbert-A)^{-1}\Vert_{\mathcal{L}(\hilbert)} (|\bar y| + L|y|)\\
		\leq & \Vert (\lambda \mathrm{I}_\hilbert-A)^{-1}\Vert_{\mathcal{L}(\hilbert)} (|\bar y| + L C |y|_{D(A)})\\
		\leq & \Vert (\lambda \mathrm{I}_\hilbert-A)^{-1}\Vert_{\mathcal{L}(\hilbert)} (|\bar y| + L C N )
		\end{split}
		\]
		where $C$ is the constant describing the compact injection of $D(A)$ in $\hilbert$. According to \cite[Corollary 2.3.3]{tw09}, one has $\Vert (\lambda \mathrm{I}_\hilbert-A)^{-1}\Vert_{\mathcal{L}(\hilbert)}\leq \frac{M}{\lambda - \omega}$ with $M$ and $\omega$ given positive numbers and $\lambda>\omega$. Therefore, one can choose $\lambda$ sufficiently large such that 
		$$
		|\mathcal{T}(\sol)|_{D(A)} \leq \frac{N}{2}. 
		$$
		Using the Schauder fixed-point theorem \cite[Theorem B.17]{c07}, one can easily deduce that the operator \eqref{prdop} is maximal. 
	}

\section*{Acknowledgements}

This work benefited from feedback from Giovanni Fantuzzi, David Goluskin, Salma Kuhlmann, Maria Infusino, Emmanuel Tr\'elat and Victor Vinnikov.

\end{document}